\documentclass[a4paper,12pt,reqno]{amsart}

\usepackage{amsmath}
\usepackage{amssymb}
\usepackage{amsfonts}
\usepackage{graphicx}
\usepackage[colorlinks]{hyperref}
\renewcommand\eqref[1]{(\ref{#1})} 

\graphicspath{ {images/} }
\title[Global existence and nonexistence of semilinear wave equation]{Global existence and nonexistence of semilinear wave equation with a new condition}

\author[Bolys Sabitbek]{Bolys Sabitbek}
\address{ \href{http://analysis-pde.org/bolys-sabitbek/}{Bolys Sabitbek:}
	\endgraf
	School of Mathematical Sciences
	\endgraf Queen Mary University of London 
	\endgraf
	United Kingdom
	\endgraf 
	and
	\endgraf 
	Al-Farabi Kazakh National University 
	\endgraf 
	Almaty,
	Kazakhstan 
	\endgraf
	{\it E-mail address} {\rm b.sabitbek@qmul.ac.uk}
}


\thanks{ The author was supported by EPSRC grant EP/R003025/2. No new data was collected or generated during the course of this research.}

\newtheoremstyle{theorem}
{10pt}          
{10pt}  
{\sl}  
{\parindent}     
{\bf}  
{. }    
{ }    
{}     
\theoremstyle{theorem}

\numberwithin{equation}{section}
\theoremstyle{plain}
\newtheorem{thm}{Theorem}[section]

\newtheorem{cor}[thm]{Corollary}
\newtheorem{lem}[thm]{Lemma}

\theoremstyle{definition}
\newtheorem{defn}[thm]{Definition}
\newtheorem{rem}[thm]{Remark}

\newtheoremstyle{defi}
{10pt}          
{10pt}  
{\rm}  
{\parindent}     
{\bf}  
{. }    
{ }    
{}     
\theoremstyle{defi}

\begin{document}
\begin{abstract}
		In this paper, we consider the initial-boundary problem for semilinear wave equation with a new condition 
		\begin{equation*}
		\alpha \int_0^{u } f(s)ds \leq uf(u) + \beta u^2 +\alpha \sigma,
		\end{equation*}
		for some positive constants $\alpha$, $\beta$, and $\sigma$, where $\beta < \frac{\lambda_1(\alpha -2)}{2}$ with $\lambda_1$ being a first eigenvalue of Laplacian. By introducing a family of potential wells, we establish the invariant sets, vacuum isolation of solutions, global existence and blow-up solutions of semilinear wave equation for initial conditions $E(0)<d$ and $E(0)=d$.
	\end{abstract}
\maketitle
\tableofcontents
\section{Introduction}
\subsection{Setting problem}
In this paper,  we consider the initial-boundary problem of semilinear wave equation 
\begin{align}\label{Wave-problem}
\begin{cases}
u_{tt} - \Delta u  = f(u), \,\,\, & x \in \Omega,\,\, t>0, \\ 
u(x,0)  =u_0(x),\,\,\, u_t(x,0)  = u_1(x) &x \in \overline{\Omega}, \\
u(x,t)=0\,\,\, & x \in \partial \Omega, t\geq 0,
\end{cases}
\end{align}
where $\Omega$ is a bounded domain in $\mathbb{R}^n$. A total energy of problem \eqref{Wave-problem} is 
\begin{align*}
E(t)&=\frac{1}{2} \int_{\Omega} u^2_t dx +\frac{1}{2} \int_{\Omega} |\nabla u|^2 dx - \int_{\Omega} [F(u) - \sigma] dx,
\end{align*}
where $F(u)=\int_0^uf(s)ds$ and $\sigma$ is some positive constant. Then a law of conservation of energy states that 
\begin{equation}
E(t)=E(0):=\frac{1}{2} \int_{\Omega} u^2_1 dx +\frac{1}{2} \int_{\Omega} |\nabla u_0|^2 dx - \int_{\Omega} [F(u_0) - \sigma] dx.
\end{equation}
A simple computation gives
\begin{align*}
\frac{dE(t)}{dt} &=\frac{d}{dt}\left[ \frac{1}{2} \int_{\Omega} u^2_t dx + \frac{1}{2} \int_{\Omega} |\nabla u|^2 dx - \int_{\Omega} [F(u) - \sigma] dx  \right]\\
& = \int_{\Omega}u_t u_{tt}dx + \int_{\Omega} \nabla u \cdot \nabla u_t dx - \int_{\Omega} f(u)u_tdx \\
& =  \int_{\Omega} [  \Delta u + f(u) ]u_t dx - \int_{\Omega} \Delta u u_t dx - \int_{\Omega} f(u)u_tdx\\
& = 0. 
\end{align*}
In this paper we use the following notations:
\begin{align*}
||u||^2 := ||  u ||^2_{L^2(\Omega)} = \int_{\Omega } |u(x,t)|^2 dx, \,\, \text{  and } \,\,
||u||^{\gamma}:= ||  u ||^{\gamma}_{L^{\gamma}(\Omega)} =\int_{\Omega } |u(x,t)|^{\gamma} dx.
\end{align*}
\subsection{A new condition on nonlinearity} 
Let $f(u)$ satisfy the following assumptions:
\begin{itemize}
	\item[(i)] $f \in C^1$, $f(0)=f'(0)=0$. 
	\item[(ii)] (a) $f$ is monotone and convex for $u>0$, concave for $u<0$; \\ 
	(b) $f$ is convex for $-\infty < u < + \infty$.
	\item[(iii)]  Let $2<\alpha \leq \gamma <2n/(n-2)$, then 
	\begin{equation}
	|uf(u)|\leq \gamma| F(u)-\sigma|,
	\end{equation}
	and
	\begin{equation}\label{cont_1}
	\alpha F(u)\leq uf(u)+\beta u^2 + \alpha \sigma,
	\end{equation}  
	where $F(u) = \int_0^u f(s) ds $, $\sigma>0$, and $\beta < \frac{\lambda_1(\alpha -2)}{2}$ with $\lambda_1$ is the first eigenvalue of Laplacian. 
	In the case (b), \eqref{cont_1} holds for $u>0$.
\end{itemize}
For our convenience, these assumptions on the nonlinearity $f(u)$ denoted as $(H)$. Also, we have two cases such $(H)-(a)$ and $(H)-(b)$, respectively. For example, $f(u)=|u|^{p-1}u$ is the case $(H)-(a)$ and $f(u)=u^p$ is the case  $(H)-(b)$.
From the assumption $(H)$ one leads the lemma that describes the properties of nonlinearity $f(u)$.
\begin{lem}
	Let $f(u)$ satisfy $(H)$. There are two cases:\\
	$(H)-(a)$ $f$ is monotone and is convex for $u>0$, concave for $u<0$, then
	\begin{itemize}
		\item[(i)] $f(u)\geq 0$ for $u>0$, and $f(u)\leq 0$ for $u<0$.
		\item[(ii)] $uf(u)\geq 0$ and $F(u)\geq 0$ for $-\infty <u<+\infty$.
		\item[(iii)] $F(u)$ is increasing on $0<u<+\infty$ and decreasing on $-\infty < u <0$.
	\end{itemize}
	$(H)-(b)$ $f$ is convex for $-\infty < u < + \infty$, then
	\begin{itemize}
		\item[(iv)] $f(u)\geq 0$ for $-\infty < u < +\infty$.
		\item[(v)] $uf(u)\geq 0$ and $F(u)\geq 0$ for $u>0$; $uf(u)\leq 0$ and $F(u)\leq 0$ for $u<0$.
		\item[(vi)] $F(u)$ is increasing on $-\infty < u < +\infty$.
	\end{itemize}
\end{lem}
Here we omit the proof.

	Note that new condition \eqref{cont_1} on nonlinearity $f(u)$ includes the following cases: 
	\begin{align}
		(2+\epsilon) F(u) &\leq uf(u), \label{eq-0.1}\\
	 (2+\epsilon) F(u) &\leq uf(u) + \sigma\label{eq-0.2}, \\
	  (2+\epsilon) F(u) &\leq uf(u) + \beta u^2 + \sigma \label{eq-0.3},
	\end{align}
where $0<\beta \leq \epsilon\lambda_1 /2$, $\sigma>0$ and $F(u)=\int_0^u f(s)ds$.	
Condition \eqref{eq-0.1} is used by Philippin and Proytcheva \cite{PP-06} to prove the blow-up solutions for the semilinear heat equation. It is a special case of an abstract condition by Levine and Payne \cite{LP2-74}. Later, Bandle and Brunner \cite{BB-98} relaxed this condition with \eqref{eq-0.2}. Recently, a new condition  \eqref{eq-0.3} was introduced by Chung and Choi  \cite{Chung-Choi} for blow-up solution of the semilinear heat equation. Moreover, blow-up solutions of nonlinear porous medium equations and pseudo-parabolic equations with the new condition are studied in \cite{ST-21, RST-21}.

Note that condition \eqref{eq-0.3} extends \eqref{eq-0.1} and \eqref{eq-0.2}. The difference between \eqref{eq-0.3} and \eqref{eq-0.2} is that \eqref{eq-0.3} depends on the domain due to term $\beta u^2$. Constant $\beta$ depends on the domain through the first eigenvalue $\lambda_1$. If the first eigenvalue $\lambda_1$ is arbitrary small then \eqref{eq-0.3} gets closer to \eqref{eq-0.2}. 

Condition \eqref{eq-0.3} is equivalent to 
\begin{equation*}
	\frac{d}{du}\left( \frac{F(u)}{u^{2+\epsilon}} - \left(\frac{\sigma}{2+\epsilon}\right) \frac{1}{u^{2+\epsilon}} - \frac{\beta}{\epsilon} \frac{1}{u^{\epsilon}} \right) \geq 0, \,\,\, u>0. 
\end{equation*}
Then it is easy to see that 
\begin{align*}
	&\eqref{eq-0.1} \,\,\, \text{ holds if and only if } \,\, F(u)=u^{2+\epsilon}h_1(u),\\
	&\eqref{eq-0.2} \,\,\, \text{ holds if and only if } \,\, F(u)=u^{2+\epsilon}h_2(u)+b,\\
	&\eqref{eq-0.3} \,\,\, \text{ holds if and only if } \,\, F(u)=u^{2+\epsilon}h_3(u)+au^2+b,
\end{align*}
for some positive constants $\epsilon$, $b$, and $a<\lambda_1/2$, where $h_1$, $h_2$, and $h_3$ are nondecreasing functions on $(0,+\infty)$. Note that the constants  $\epsilon$, $b$, and $a$ shall be different in each case. 

\begin{lem}\label{lem}
	Let $f$ be a function satisfying $\eqref{eq-0.3}$ and $uf(u)\geq \lambda u^2$ for $u>0$ along with $\lambda > \lambda_1$. Then condition $\eqref{eq-0.3}$ implies that there exists $m>0$ such that $h_3(u)>0$ for $u>m$. In this case, it is possible to find $\mu >0$ such that $uf(u)\geq \mu u^{2+\epsilon}$ for $u\geq m$. It means that conditions \eqref{eq-0.2} and \eqref{eq-0.3} are equivalent. 
\end{lem}
\begin{proof}[Proof of Lemma \ref{lem}]
	Recall the facts $0<a<\lambda_1/2$ and $\lambda > \lambda_1$ that allows to obtain
	\begin{equation*}
		F(u) \geq \frac{\lambda}{2}u^2 > \frac{\lambda_1}{2}u^2,
	\end{equation*}
	then 
	\begin{equation*}
		u^{2+\epsilon}h_3(u) = F(u) -au^2 -b \geq \frac{\lambda -\lambda_1}{2} u^2 -b,
	\end{equation*}
	which goes to $+\infty$ as $u\rightarrow +\infty$. So there is $m>1$ such that $h_3(m)>0$, which implies that 
	\begin{equation*}
		F(u) \geq u^{2+\epsilon}h_3(u), \,\,\, u\geq m.
	\end{equation*}
	Applying it into condition \eqref{eq-0.3}, we get
	\begin{align*}
		u^{2+\epsilon} h_3(m) &\leq uf(u) + \beta u^2 + \sigma \\
		& \leq (1+ \epsilon/2)uf(u) + \sigma, \,\,\, u\geq m>1.
	\end{align*}
	That gives 
	\begin{equation*}
		uf(u) \geq \mu u^{2+\epsilon}, \,\, u\geq m>1,
	\end{equation*}
	for some constant $\mu>0$.
	
	Next we show that \eqref{eq-0.2} and \eqref{eq-0.3} are equivalent. Since $0<\beta \leq \epsilon \lambda_1/2$ and $uf(u)\geq \lambda u^2>\lambda_1 u^2$ for $u>0$, it follows from \eqref{eq-0.3} that
	\begin{equation*}
		\epsilon_1 F(u) + (2+\epsilon_2)F(u) \leq uf(u) + \frac{\epsilon \lambda_1}{2} u^2 + \sigma,
	\end{equation*}
	where $\epsilon_1= \frac{\epsilon \lambda_1}{\lambda}>0$ and $\epsilon_2 = \epsilon -\epsilon_1>0$. Then we have
	\begin{align*}
		uf(u) +\sigma &\geq (2+\epsilon_2)F(u) + \epsilon_1 \int_0^u [f(s) -\lambda s]ds\\
		&\geq (2+\epsilon_2)F(u),
	\end{align*}
	which is \eqref{eq-0.2}. This completes the proof.
\end{proof}

\subsection{Literature overview}
There are huge amount of papers devoted to problem \eqref{Wave-problem}, and it is impossible to cover everyone. We focus on a study of long-term behavior of evolution equations.  Early works of this direction are by Keller \cite{Keller-57}, Sattinger \cite{Sattinger-68}, J\"orgens \cite{Jorgens-70}, Tsutsumi \cite{Tsutsumi-1, Tsutsumi-2, Tsutsumi-3}, Levine \cite{Levine73} and others, where they investigated conditions on the initial data and source term $f$ for which solutions of \eqref{Wave-problem} blow up in a finite time. In 1975, Payne-Sattinger \cite{PS-75} introduced a potential wells method to study problem \eqref{Wave-problem}, and its extensions to parabolic and abstract operator equations. The potential wells method gives a better intuitive understanding of the phenomena of instability, because we could get this intuition by considering one-dimensional mechanical analogue of problem \eqref{Wave-problem} as $\ddot{x} = -x + f(x)$. It becomes one of the most powerful tool for studying nonlinear evolution equations, (see more in the book \cite{QS-book}).

In 2002, Vitillaro extended the potential wells theory for the wave equation with an internal nonlinear source and a nonlinear boundary damping as follows
\begin{align*}
\begin{cases}
u_{tt} - \Delta u  = |u|^{p-2}, \,\,\, & x \in \Omega,\,\, t>0, \\ 
u(x,0)  =u_0(x),\,\,\, u_t(x,0)  = u_1(x) &x \in \overline{\Omega}, \\
u=0\,\,\, & \text{  on }  \,\, \Gamma_0, t\geq 0,\\
\frac{\partial u}{\partial \nu } = \alpha(x) |u_t|^{m-2}u_t &  \text{  on }  \,\, \Gamma_1, t\geq 0,
\end{cases}
\end{align*}
where $\Omega \subset \mathbb{R}^n$ is a regular bounded domain with smooth boundary given as the disjoint union $\partial \Omega= \Gamma_0 \cup \Gamma_1$. 

In 2006, Liu-Zhao \cite{LZh-06} generalized and improved the results of Payne-Sattinger \cite{PS-75}, by obtaining the invariant sets and vacuum isolation of solutions, and a threshold result of global existence and nonexistence of solutions with condition $(H)-(a)$ on nonlinearity $f(u)$, (see also \cite{LZh-03} and \cite{LZh-04}). Later, Runzhang Xu \cite{Xu-10} complemented the result of Liu-Zhao \cite{LZh-06} by considering the critical initial energy $E(0)=d$. We refer other works of Runzhang Xu and his collaborators in this direction \cite{XS-13}, \cite{XZh-17}, and \cite{XLN-20}.  

In 2020, Chen-Chen proved the unsolved case $(H)-(b)$ which fills important gaps regarding problem \eqref{Wave-problem}. By introducing a family of potential wells, the invariant sets and vacuum isolation of solutions were obtained. Moreover, the global nonexistence of solution and asymptotic behavior of solution  were proved for semilinear parabolic equations.

In addition to the references given in detail above, the literature is rich with results on wave equations along with a damped term. Many pioneering papers such as Lions and Strauss \cite{LS-65}, Nakao and  Nanbu \cite{NN-75}, as well as works by Glassey \cite{Glassey} and Ball \cite{Ball-77} are worthy of mention. Moreover,  we would like to mention \cite{CCM-04, CZT-20, DSS-20, Esuivel-03, GS-06, Ikehata-96, LX-20, MRV-18, NO-93} and references therein.

\subsection{Main goal of the paper} A main aim of this work is to study a global existence and blow-up of solutions to initial-boundary problem \eqref{Wave-problem} with a new condition on a nonlinearity $f(u)$ as follows
\begin{equation*}
\alpha\int_0^u f(s)ds \leq uf(u) + \beta u^2 +\alpha \sigma,
\end{equation*}
for some positive constants $\alpha$, $\beta$, and $\sigma$, where $\beta < \frac{\lambda_1(\alpha -2)}{2}$ with $\lambda_1$ being the first eigenvalue of Laplacian. The new condition was introduced by Chung-Choi \cite{Chung-Choi} for a heat equation. By introducing a family of potential wells, we are able to establish the invariant sets, vacuum isolation of solutions, and a threshold results of global existence and nonexistence of solutions with the new condition on a source term. Moreover, critical initial conditions $E(0)=d$ are discussed for global existence and blow-up solutions. 
\subsection{Plan of the paper}
In Section \ref{Sec2}, we introduce the potential wells theory along with its properties. In Section \ref{Sec3}, we use the family of potential wells to obtain invariant sets and vacuum isolating solutions of problem \eqref{Wave-problem}. In Section \ref{Sec4} and \ref{Sec5}, we prove the global existence and nonexistence of solutions for problem \eqref{Wave-problem} with initial conditions $E(0)<d$ and $E(0)=d$. In Section \ref{Sec6}, we provide a brief conclusion of results.

\section{Potential Wells Theory}\label{Sec2}
In this section, we introduce a family of potential wells. First, we would like to illustrate idea of potential wells by giving a simple example such as one-dimensional mechanical analogue of problem \eqref{Wave-problem} 
\begin{equation}\label{eq-1D}
\ddot{x} = -x + f(x),
\end{equation}
where $x$ is a real number. This example was given by Payne-Sattinger \cite{PS-75} and it describes a mechanical system with one degree of freedom, while problem  \eqref{Wave-problem} may be considered as a system with an infinite number of degree of freedom.
The potential energy of equation \eqref{eq-1D} is defined by 
\begin{equation*}
	V(x) = \frac{x^2}{2} - F(x),
\end{equation*}
where $F(x) = \int_0^x f(s)ds$. Imagine that potential energy $V(x)$ has the qualitative shape with a local minimum at $x=0$ and a local maximum at $x = x_*$.

The set 
\begin{equation*}
	W = \{ x: V(x)<d, \,\, x < x_* \}
\end{equation*} 
describes a potential well. The total energy of equation \eqref{eq-1D} is 
\begin{equation*}
	E(t) = \frac{\dot{x}^2}{2} + \frac{x^2}{2} - F(x)= \frac{\dot{x}^2}{2} + V(x).
\end{equation*}
As we know that the total energy of \eqref{eq-1D} is conserved under the motion. Then by the conservation of energy we have 
\begin{itemize}
	\item If $E(0)<d$ and $x(0)\in W$, then $x(t)\in W$ for all $t>0$.
	\item If $E(0)<d $ and $x(0)>x_*$, then $x(t) > x_* $ for all $t>0$.
\end{itemize} 
In order to cross into potential well $W$, the total energy would have to be greater than the depth of potential well $d$.

This idea can be extended to the infinite dimensional case \eqref{Wave-problem}. The \textbf{potential energy} associated with \eqref{Wave-problem} is the functional 
\begin{align*}
	J(u) & = \frac{1}{2} || \nabla u ||^2 - \int_{\Omega} [F(u)-\sigma] dx.
\end{align*}
We define the \textbf{Nehari functional} by
\begin{equation*}
	I(u)  = || \nabla u ||^2 - \int_{\Omega } uf(u)dx, \,\,\, u \in H_0^1(\Omega),
\end{equation*}
where $I(u)$ can be considered as the extension of point $x_*$ in one dimension analogue. 

The depth of potential wells is defined by
\begin{equation*}
		d := \inf \{J(u): u \in H_0^1(\Omega)\backslash \{ 0 \}, I(u)=0\}.
\end{equation*}
Also the quantity $d$ can be understood as a mountain pass energy. For instance, there is the valley. If we pour water into the valley containing a local minimum at the origin, as the water level rises in the direction of the potential energy and passes at a local maximum $d$. 

Note that in all proofs, we should treat slightly differently the cases $(a)$ and $(b)$ for assumption $(H)$ for nonlinearity $f(u)$.  
\begin{lem}\label{lem-1}
	Let $f(u)$ satisfy $(H)$, then we have 
	\begin{itemize}
		\item[(a)] $uf'(u) - f(u) \geq 0$ holds in the case $(H)-(a)$ with equality holding only for $u=0$, and in the case $(H)-(b)$ for $u \geq 0$.
		\item[(b)] $|F(u)-\sigma|\leq A |u|^{\gamma}$ for some $A>0$ and all $u\in R$.
		\item[(c)] $F(u)-\sigma \geq B |u|^{\lambda}$ for $|u|\geq 1$ with $B = \frac{F(\widetilde{u})}{\widetilde{u}^{\lambda}}>0$ and $\lambda = \frac{\alpha}{1+\beta}>2$.
	\end{itemize}
\end{lem}
\begin{proof}[Proof of Lemma \ref{lem-1}]
$(a)$ If $f(u)$ satisfies $(H)-(a)$, then the proof is based on the second variation of functional $J$ at a critical point $w$. Let 
\begin{equation*}
	i(\tau):= J(w+\tau v) = \frac{1}{2} \int_{\Omega } |\nabla (w+ \tau v)|^2 dx - \int_{\Omega }F(w+\tau v)dx, 
\end{equation*} 
where $v \in C_0^{\infty}(\Omega)$. The second variation of $J$ is
\begin{equation*}
	i''(0) = \int_{\Omega } |\nabla v|^2 dx - \int_{\Omega } f'(w)v^2 dx.
\end{equation*}
If $w$ is a local minimum of functional $J$, then $i''(0)\geq 0$. Then $w=0$ gives $i''(0) = \int_{\Omega } |\nabla v|^2dx\geq 0$. On other hand, $w$ is a non-trivial critical point, then $i''(0)\leq 0$. Let $w$ be an admissible test function, then we compute
\begin{align*}
i''(0) &= \int_{\Omega } |\nabla w|^2 dx - \int_{\Omega } f'(w)w^2 dx\\
	& = - \int_{\Omega } w(\Delta w + f'(w)w)dx \\
	& = -\int_{\Omega } w( f'(w)w - f(w))dx,
\end{align*}
it follows that $ w( f'(w)w - f(w)) \geq 0$. 

In the case $(H)-(b)$ $, $ $f$ is convex for $-\infty < u < + \infty$, this gives $f(0)\geq f(u)+f'(u)(0-u)$ that gives $uf'(u)\geq f(u)$.
 
 $(b)$ From $(H)$, the growth condition $|uf(u)|\leq \gamma |F(u)-\sigma|$ implies
 \begin{align*}
 	\frac{|f(u)|}{|F(u)-\sigma|} \leq \frac{\gamma}{|u|} \,\,\, \text{ for } \,\,\, u \neq 0.
 \end{align*}
 Taking any known fixed $\widetilde{u} \neq 0$, we integrate above inequality from $\widetilde{u}$ to $u$, which gives
 \begin{align*}
 	\int_{\widetilde{u}}^u \frac{|f(s)|}{|F(s)-\sigma|} ds &\leq 	\int_{\widetilde{u}}^u \frac{\gamma}{|s|}ds,
 \end{align*} 
 and by using $F'(u)=f(u)$, then we have
 \begin{align*}
 	\ln \frac{|F(u)-\sigma|}{|F(\widetilde{u})-\sigma|}\leq \ln \frac{|u|^{\gamma}}{|\widetilde{u}|^{\gamma}},
 \end{align*}
 that is $|F(u)-\sigma| \leq A |u|^{\gamma}$ where $A = \frac{|F(\widetilde{u})-\sigma|}{|\widetilde{u}|^{\gamma}}$.
 
 $(c)$ From $(H)$ we have $\alpha F(u) \leq uf(u) + \beta u^2 +\alpha  \sigma$. For $|u|\geq 1$, we write
 \begin{align*}
 	\frac{\alpha}{|u|}  \leq  \frac{|f(u)|}{F(u)-\sigma} +  \frac{\beta |u|}{F(u)-\sigma}  \leq (1+\beta)  \frac{|f(u)|}{F(u)-\sigma},
 \end{align*}
 then we get
 \begin{equation}
 	\frac{|f(u)|}{F(u)-\sigma} \geq \frac{\lambda}{|u|} ,
 \end{equation}
 where $\lambda= \frac{\alpha}{1+\beta}$. As the previous approach on $(b)$, we get $F(u)-\sigma \geq B |u|^{\lambda} $ where $B = \frac{F(\widetilde{u})}{\widetilde{u}^{\lambda}}$. 
\end{proof}

\begin{rem}
	Note that from Lemma \ref{lem-1} and condition $(H)$ we have the following properties:
	\begin{itemize}
		\item $|f(u)|\leq \gamma A|u|^{\gamma-1}$ for all $u \in R$.
		\item $uf(u) + \beta u^2 \geq \alpha B|u|^{\lambda}$ for $u\geq 1$.
	\end{itemize}
\end{rem}

\begin{lem}\label{lem-2}
	Let $f(u)$ satisfy $(H)$. Let $u\in H_0^1(\Omega)$, $||\nabla u ||\neq 0$, and 
		\begin{align}
		\phi(\epsilon) = \frac{1}{\epsilon} \int_{\Omega } u f(\epsilon u) dx.
	\end{align}
	Then we have the following properties of function $\phi$: 
	\begin{itemize}
		\item $\phi(\epsilon)$ is increasing on $0< \epsilon < \infty$.
		\item $\lim_{\epsilon \rightarrow 0}\phi(\epsilon)=0$ and $\lim_{\epsilon \rightarrow +\infty}\phi(\epsilon)=+\infty$.
	\end{itemize}
In the case $(H)-(b)$, there is additional assumption $u(x)> 0$.
\end{lem}
\begin{proof}[Proof of Lemma \ref{lem-2}]
	Let 
	\begin{align*}
		\frac{d\phi(\epsilon)}{d\epsilon} &= \frac{1}{\epsilon^2} \int_{\Omega } [\epsilon u^2 f'(\epsilon u) - uf(\epsilon u)]  dx \\
		&= \frac{1}{\epsilon^3} \int_{\Omega } \epsilon u[\epsilon u f'(\epsilon u) - f(\epsilon u)]  dx >0.
	\end{align*}
	By $(a)$ in Lemma \ref{lem-1}, we see that $\phi(\epsilon)$ is increasing on $0<\epsilon < \infty$.
	
	 The property $\lim_{\epsilon \rightarrow 0}\phi(\epsilon)=0$ comes from the fact that
	 \begin{align*}
	 	|\phi (\epsilon)| \leq \frac{1}{\epsilon^2} \int_{\Omega } |\epsilon u f(\epsilon u)|dx \leq \frac{\gamma A}{\epsilon^2} \int_{\Omega } |\epsilon u|^{\gamma}dx = \gamma \epsilon^{\gamma -2} A || u ||^{\gamma}_{\gamma}.
	 \end{align*}
	 
	  The property $\lim_{\epsilon \rightarrow +\infty}\phi(\epsilon)=+\infty$ can be shown by using the fact $uf(u)\geq 0$ and  $(c)$ in Lemma \ref{lem-1} as follows
	 \begin{align*}
	 	\phi(\epsilon) &= \frac{1}{\epsilon^2}\int_{\Omega } \epsilon u f(\epsilon u) dx \geq \frac{1}{\epsilon^2}\int_{\Omega_{\epsilon} } \epsilon u f(\epsilon u) dx\\
	 	& \geq \frac{1}{\epsilon^2}\int_{\Omega_{\epsilon} } [\alpha F(\epsilon u) - \beta \epsilon^2 u^2 -  \alpha \sigma  ]dx \\
	 	& \geq \alpha B \epsilon^{\lambda -2} \int_{\Omega_{\epsilon} } |u|^{\lambda} dx - \beta \int_{\Omega_{\epsilon} } u^2 dx,
	 \end{align*}
	 where $\Omega_{\epsilon} = \{ x \in \Omega \,\, | \,\, u(x)\geq \frac{1}{\epsilon} \}$ and
	 \begin{equation*}
	 	\lim_{\epsilon \rightarrow + \infty} \int_{\Omega_{\epsilon} } |u|^{\lambda} dx = || u ||^{\lambda}_{\lambda} >0.
	 \end{equation*}
	 This completes the proof.
\end{proof}
When nonlinearity term $f(u)$ satisfies condition $(H)$, then all non-trivial critical points are a priori unstable equilibria for semilinear wave equation \eqref{Wave-problem}. The next lemma will provide the local minimum of $J$, a critical point $\epsilon^*$ and positiveness of a depth of potential well. 
\begin{lem}\label{lem-3} 	Let $f(u)$ satisfy $(H)$, $u\in H_0^1(\Omega)$, and $||\nabla u ||\neq 0$. Then 
	\begin{itemize}
		\item[(i)] $ \lim_{\epsilon \rightarrow 0}J(\epsilon u)=0$.
		\item[(ii)] $ \frac{d}{d\epsilon} J(\epsilon u) = \frac{1}{\epsilon} I(\epsilon u)$.\\
		In addition, we assume $u(x)>0$ in the case $(H)-(b)$.
		\item[(iii)] $ \lim_{\epsilon \rightarrow + \infty}J(\epsilon u)= -\infty$.
		\item[(iv)] There exists a unique $\epsilon^*=\epsilon^*(u)>0$ such that
		\begin{equation}
			\frac{d}{d\epsilon} J(\epsilon u)|_{\epsilon=\epsilon^*} = 0.
		\end{equation}
		\item[(v)] $J(\epsilon u)$ is increasing on $0 \leq \epsilon \leq \epsilon^*$, decreasing on $\epsilon^* \leq \epsilon < \infty$ and takes the maximum at $\epsilon = \epsilon^*$.
		\item[(vi)] $I(\epsilon u)>0$ for $0<\epsilon < \epsilon^*$, $I(\epsilon u)<0$ for $\epsilon^*<\epsilon < \infty$  and $I(\epsilon^*u)=0$.
	\end{itemize}
\end{lem}
\begin{proof}[Proof of Lemma \ref{lem-3}]
	The proof of (i) follows from $|F(u)-\sigma|\leq A |u|^{\gamma}$ and  
	\begin{equation*}
		J(\epsilon u) = \frac{\epsilon^2}{2} || \nabla u||^2 - \int_{\Omega } [F(\epsilon u) -\sigma ]dx.
	\end{equation*}
	The proof of (ii) follows from
	\begin{equation*}
		 \frac{d}{d\epsilon} J(\epsilon u) = \epsilon || \nabla u||^2 - \int_{\Omega } uf(\epsilon u)dx = \frac{1}{\epsilon}I(\epsilon u).
	\end{equation*}
 First, we consider the case (iii) for $(H)-(a)$. Let 
 \begin{align*}
 		J(\epsilon u) &= \frac{\epsilon^2}{2} || \nabla u||^2 - \int_{\Omega } [F(\epsilon u) -\sigma] dx\\
 		&\leq  \frac{\epsilon^2}{2} || \nabla u||^2 - B \epsilon^{\lambda} \int_{\Omega_{\epsilon} } |u|^{\lambda}dx.
 \end{align*}
 Since $\lambda >2$, $J(\epsilon u)$ goes to $-\infty$ for $\epsilon \rightarrow +\infty$. This guaranties the existence of a critical point $\epsilon^*$. The case $(H)-(b)$ can be proved in the same way with additional assumption $u>0$.
 
 The case (iv), the uniqueness of critical point, can be proved by supposing that there two roots of equation $\frac{dJ(\epsilon u)}{d\epsilon}=0$ as $\epsilon_1<\epsilon_2$. Then we have
 \begin{equation*}
 	 \epsilon_1 || \nabla u||^2 - \int_{\Omega } uf(\epsilon_1 u)dx =0,
 \end{equation*}
 and 
  \begin{equation*}
 \epsilon_2 || \nabla u||^2 - \int_{\Omega } uf(\epsilon_2 u)dx =0.
 \end{equation*}
 Elimination of $|| \nabla u||^2 $ from above equations gives
 \begin{equation*}
 	\int_{\Omega } u \left[\frac{f(\epsilon_2 u)}{\epsilon_2}- \frac{f(\epsilon_1 u)}{\epsilon_1} \right]dx = 0.
 \end{equation*}
 This expression can be rewritten by the following substitution $w=\epsilon_1 u$ and $\epsilon = \epsilon_2/\epsilon_1>1$ as
 \begin{equation}\label{eq-2.4}
 	\phi(\epsilon)-\int_{\Omega } wf(w)=0.
 \end{equation}
 It is easy to see from Lemma \ref{lem-2} that equation \eqref{eq-2.4} does not hold.
 
 The case (v) and (vi) follow from Lemma \ref{lem-2} and 
 \begin{equation*}
 	\frac{d}{d\epsilon} J(\epsilon u) = \frac{1}{\epsilon} I(\epsilon u)= \epsilon ( || \nabla u||^2 - \phi(\epsilon)).
 \end{equation*}
 The proof is completed.
\end{proof}
\begin{rem} In this remark, we would like to show how the notation of $I(u)$ is adopted.
	
	Note that for a unique $\epsilon^*=\epsilon^*(u)>0$ we have
	\begin{equation*}
	\frac{d}{d\epsilon} J(\epsilon u)|_{\epsilon=\epsilon^*} = 0,
	\end{equation*}
and $J(\epsilon^*u)$ can be considered as the highest level attained when leaving the {\em potential well} $W$ along a ray in the direction $u$. Now the depth $d$ of potential well is defined by
\begin{equation*}
	d = \inf_{u\neq 0} J(\epsilon^*u).
\end{equation*}
Then $u$ can be normalized so that $\epsilon^*=1$, which gives
\begin{equation*}
	I(u)\equiv \frac{d}{d\epsilon^*} J(\epsilon^* u)|_{\epsilon^*=1} = || \nabla u||^2 - \int_{\Omega } uf(u)dx.
\end{equation*}
Then we have 
\begin{equation*}
d := \inf \{J(u): u \in H_0^1(\Omega)\backslash \{ 0 \}, I(u)=0\}.
\end{equation*} 
\end{rem}

In order to introduce the family of potential wells $\{W_{\delta}\}$ and  $\{V_{\delta}\}$, we define
 \begin{align*}
	I_{\delta}(u) = \delta || \nabla u||^2 - \int_{\Omega } uf(u)dx, \,\,\, \delta >0,
\end{align*}
and 
\begin{equation*}
d(\delta) := \inf \{J(u): u \in H_0^1(\Omega)\backslash \{ 0 \}, I_{\delta}(u)=0\}.
\end{equation*}

The following lemmas will describe how space $H_0^1(\Omega)$ is divided into two parts and the properties of inside and outside parts are discussed. For a graphical illustration, we will refer to Figure 1 in Liu-Zhao's paper \cite{LZh-06}.
\begin{lem}\label{lem-4}
	Let $f(u)$ satisfy $(H)$. Suppose that $u \in H_0^1(\Omega)$ and
		\begin{align*}
	r(\delta) = \left( \frac{\delta}{a C^{\gamma}_*} \right)^{\frac{1}{\gamma-2}} \,\,\, \text{and} \,\,\, a = \sup \frac{u f(u)}{|u|^{\gamma}},
	\end{align*}
	where $C_*$ is the embedding constant form $H_0^1(\Omega)$ into $L^{\gamma}(\Omega)$. Then we have
	\begin{itemize}
		\item[(i)] If $0<|| \nabla u || < r(\delta)$, then $I_{\delta}(u)>0$. In the case $\delta =1$, if $0< || \nabla u ||  <r(1)$, then $I(u)>0$.
		\item[(ii)] If $I_{\delta}(u)<0$, then $|| \nabla u || >r(\delta)$. In the case $\delta =1$, if $I(u)<0$, then $|| \nabla u || >r(1)$.
		\item[(iii)] If $I_{\delta}(u)=0$, then $|| \nabla u || \geq r(\delta)$. In the case $\delta =1$, if $I(u)=0$, then $|| \nabla u || \geq r(1)$.
	\end{itemize}
\end{lem}
Note that $(i)$ says that if $||\nabla u||$ is inside a ball of radius $r(\delta)$, then one lies inside part of space $H_0^1(\Omega)$. $(ii)$ says that if one lies outside part of space $H_0^1(\Omega)$ then $||\nabla u||$ is outside a ball of radius $r(\delta)$. $(iii)$ says that one lies on the boundary then $||\nabla u||$ is on surface or outside of a ball with radius $r(\delta)$.
\begin{proof}[Proof of Lemma \ref{lem-4}]

(i)	From $0<|| \nabla u ||<r{(\delta)}$ we have
	\begin{align*}
		\int_{\Omega } u f(u) dx &\leq \int_{\Omega } |u f(u)| dx \leq a \int_{\Omega } |u|^{\gamma} dx = a || u||_{\gamma}^{\gamma}\\
		&\leq a C_*^{\gamma} || \nabla u||^{\gamma-2}|| \nabla u||^{2}< \delta || \nabla u||^2,
	\end{align*} 
	which gives $I_{\delta}(u)>0$.
	
	(ii) From $I_{\delta}(u)<0$ we get
	\begin{equation*}
		\delta || \nabla u ||^2 < \int_{\Omega } u f(u) dx \leq a || u ||^{\gamma}_{\gamma} \leq aC_*^{\gamma} || \nabla u||^{\gamma -2}  || \nabla u||^{2}, 
	\end{equation*}
	which implies $||\nabla u || >r(\delta)$. 
	
	(iii)  From $I_{\delta}(u)=0$ we obtain
	\begin{equation*}
	\delta || \nabla u ||^2 \leq \int_{\Omega } u f(u) dx \leq a || u ||^{\gamma}_{\gamma} \leq aC_*^{\gamma} || \nabla u||^{\gamma -2}  || \nabla u||^{2}, 
	\end{equation*}
	which implies $||\nabla u || \geq r(\delta)$.  
\end{proof}
Next lemma describes the properties of $d(\delta)$ for a family of potential wells.
\begin{lem}\label{lem-5}
	Let $f(u)$ satisfy $(H)$. Then 
	\begin{itemize}
		\item[(i)] $d(\delta)>a(\delta)r^2(\delta)$ for $0<\delta< \frac{\alpha}{2} -\frac{\beta}{\lambda_1}$, where $a(\delta)= \frac{1}{2} - \frac{\delta}{\alpha} - \frac{\beta}{\lambda_1\alpha}$.
		\item[(ii)] $\lim_{\delta \rightarrow 0}d(\delta)=0$ and there exists a unique $b$, $\frac{\alpha}{2} - \frac{\beta}{\lambda_1} \leq b \leq \frac{\gamma }{2}$ such that $d(b)=0$ and $d(\delta)>0$ for $0<\delta<b$.
		\item[(iii)] $d(\delta)$ is strictly increasing on $0\leq \delta \leq 1$, strictly decreasing on $1\leq \delta \leq \infty$ and takes the maximum $d=d(1)$ at $\delta =1$.
	\end{itemize}
\end{lem}
\begin{proof}[Proof of Lemma \ref{lem-5}]
	We begin the proof by considering case $(H)-(a)$.
	
	(i) If $I_{\delta}(u) =0$ and $|| \nabla u||\neq 0$, then we have $|| \nabla u|| \geq r(\delta)$ from Lemma \ref{lem-4}. By applying $I_{\delta}(u)=0$, the Poincar\'e inequality, and Lemma \ref{lem-4}, we have
	\begin{align*}
		J(u) &= \frac{1}{2} || \nabla u ||^2 - \int_{\Omega } [F(u)-\sigma] dx\\
		     &\geq  \frac{1}{2} || \nabla u ||^2 - \frac{1}{\alpha} \int_{\Omega }uf(u)dx- \frac{\beta}{\alpha} \int_{\Omega }u^2dx  \\
		     & \geq \left[\frac{1}{2}  - \frac{\delta}{\alpha} - \frac{\beta}{\lambda_1\alpha} \right]|| \nabla u ||^2\\
		     & \geq a(\delta) r^2(\delta),
	\end{align*}
	where $0<\delta< \frac{\alpha}{2} -\frac{\beta}{\lambda_1}$. For the case $(H)-(b)$, the proof is same with additional condition $u>0$.
	
	$(ii)$ Lemma \ref{lem-2} implies that for $u \in H_0^1(\Omega)$, $||\nabla u||\neq 0$ and any $\delta >0$ we can define a unique $\epsilon(\delta) = \phi^{-1}(\delta || \nabla u||^2)$ such that
	\begin{equation}\label{eq-3}
		\epsilon^2\phi(\epsilon) = \int_{\Omega } \epsilon u f(\epsilon u) dx=\delta || \nabla (\epsilon u)||^2 .
	\end{equation}
	Then for $I_{\delta}(\epsilon u)=0$ we get
	\begin{equation*}
		\lim_{\delta \rightarrow 0} \epsilon (\delta) =0 \,\, \text{ and } \,\, \lim_{\delta \rightarrow +\infty} \epsilon (\delta) = +\infty.
	\end{equation*}
	From Lemma \ref{lem-3} it follows that 
	\begin{align*}
		\lim_{\delta \rightarrow 0} J(\epsilon u) = \lim_{\epsilon \rightarrow 0}J(\epsilon u)= 0  \,\, &\text{ and } \,\, \lim_{\delta \rightarrow 0} d(\delta) = 0, \\
		\lim_{\delta \rightarrow +\infty} J(\epsilon u) = \lim_{\epsilon \rightarrow +\infty}J(\epsilon u)= -\infty  \,\, &\text{ and } \,\, \lim_{\delta \rightarrow + \infty} d(\delta) = -\infty.
	\end{align*} 
	Existence of $d(b)=0$ for $b\geq \alpha/2 -\lambda_1 \beta$ and $d(\delta)>0$ for $0<\delta<b$ follows from the part $(i)$ and above expressions. Upper bound of $b\leq \gamma/2$ follows from the facts $|uf(u)|\leq \gamma |F(u)-\sigma|$ and $I_{\delta}(u)=0$, then
	\begin{align*}
		J(u) &= \frac{1}{2} || \nabla u||^2  - \int_{\Omega } [F(u)-\sigma]dx \\
		& \geq \frac{1}{2} || \nabla u||^2 - \gamma^{-1} \int_{\Omega } uf(u)dx\\
		& = \left( \frac{1}{2} - \frac{\delta}{\gamma}\right) || \nabla u ||^2 < 0, \,\,\, \text{ for } \,\,\, \delta > \gamma/2.
	\end{align*}
In the case $(H)-(b)$, by assumption $u(x)> 0$ we can prove in the same way.

For the case $(iii)$, it is enough to show that $d(\delta')<d(\delta'')$ for any $0<\delta'<\delta''<1$ or $1<\delta''<\delta' <b$. In other word, we need to prove that for some $c(\delta',\delta'')>0$, we have 
$$J(v)>J(u)-c(\delta',\delta''),$$
for any $u\in H_0^1(\Omega)$ and $v\in H_0^1(\Omega)$ with $I_{\delta''}(u)=0$, $|| \nabla u|| \neq 0$ and $I_{\delta'}(u)=0$, $|| \nabla v|| \neq 0$, respectively.

First, we define $\epsilon(\delta)$ for $u$ as \eqref{eq-3}, then $I_{\delta}(\epsilon(\delta)u)=0$ and $\epsilon(\delta'')=1$. Let us have $g(\epsilon)=J(\epsilon u)$, then
\begin{align*}
	\frac{d}{d\epsilon} g(\epsilon) = \frac{1}{\epsilon} \left( || \nabla (\epsilon u) ||^2 - \int_{\Omega } \epsilon u f(\epsilon u) dx \right) = (1-\delta )\epsilon ||\nabla u ||^2,
\end{align*}
where we use the fact $\delta || \nabla u ||^2 = \int_{\Omega } uf(u)dx$. 

Now we choose $v=\epsilon(\delta')u$, then $I_{\delta'}(v)=0$ and $|| \nabla v || \neq 0$. For $0<\delta'<\delta''<1$, we get
\begin{align*}
	J(u) - J(v) &= J(\epsilon(\delta'') u) - J(\epsilon(\delta') u) \\
	&= g(1) - g(\epsilon(\delta')) \\
	&= (1-\delta')\epsilon(\delta')|| \nabla u||^2(1-\epsilon(\delta'))\\
	&> (1-\delta'')\epsilon(\delta')r^2(\delta'')(1-\epsilon(\delta')) \equiv c(\delta',\delta'').
\end{align*}
For $1<\delta''<\delta'<b$, we have
\begin{align*}
		J(u) - J(v) &= J(\epsilon(\delta'') u) - J(\epsilon(\delta') u)\\
			&= g(1) - g(\epsilon(\delta')) \\
			&> (\delta'' -1)r^2(\delta'')\epsilon(\delta'') (\epsilon(\delta'')-1).
\end{align*}
This completes the proof.
\end{proof}
\begin{lem}\label{lem-6}
	Let $f(u)$ satisfy $(H)$ and $0<\delta<\frac{\alpha}{2}-\frac{\beta}{\lambda_1}$. Then the following properties hold 
	\begin{itemize}
		\item[(i)] Suppose that $J(u)\leq d(\delta)$ and $I_{\delta}(u)>0$, then
		\begin{equation*}
		0<|| \nabla u||^2 < \frac{d(\delta)}{a(\delta)}.
		\end{equation*}
		In the case $\delta=1$, if $J(u)\leq d$ and $I(u)>0$, then 
			\begin{equation*}
		0<|| \nabla u||^2 < \frac{d}{a(1)}.
		\end{equation*}
		\item[(ii)] Suppose that $J(u)\leq d(\delta)$ and $I_{\delta}(u)=0$, then 
			\begin{equation*}
		0<|| \nabla u||^2 < \frac{d(\delta) }{a(\delta)}.
		\end{equation*}
		In the case $\delta=1$, if $J(u)\leq d$ and $I(u)=0$, then 
			\begin{equation*}
		0<|| \nabla u||^2 < \frac{d}{a(1)}.
		\end{equation*}
		\item[(iii)] Suppose that $J(u)\leq d(\delta)$ and $|| \nabla u||^2>d(\delta)/a(\delta)$, then $I_{\delta}(u)<0$. \\
			In the case $\delta=1$, if $J(u)\leq d$ and $|| \nabla u||^2>{d}/{a(1)}$, then $I(u)<0$. 
	\end{itemize}
\end{lem}
\begin{proof}[Proof of Lemma \ref{lem-6}] The case $(i)$ follows from 
		\begin{align*}
 J(u) &= \frac{1}{2} || \nabla u ||^2 - \int_{\Omega } [F(u)- \sigma] dx\\
	&\geq  \frac{1}{2} || \nabla u ||^2 - \frac{1}{\alpha} \int_{\Omega }uf(u)dx- \frac{\beta}{\alpha} \int_{\Omega }u^2dx  \\
	& > \left[\frac{1}{2}  - \frac{\delta}{\alpha} - \frac{\beta}{\lambda_1\alpha} \right]|| \nabla u ||^2  \\
	& = a(\delta) || \nabla u ||^2 .
	\end{align*}
	The proof of cases $(ii)$ and $(iii)$ follows from the similar argument. The proofs hold for the case $(H)-(b)$ with additional assumption $u(x)>0$.
\end{proof}

Then a family of potential wells can be defined for $0<\delta<b$ as follows
\begin{align*}
	W_{\delta} &:= \{ u \in H_0^1(\Omega) \,\,| \,\, I_{\delta}(u)>0, J(u)<d(\delta) \} \cup \{0\},\\
	\overline{W}_{\delta} & := W_{\delta}\cup \partial W_{\delta}=\{ u \in H_0^1(\Omega) \,\,| \,\, I_{\delta}(u)\geq 0, J(u) \leq d(\delta) \},
\end{align*}
In the case $\delta=1$, we have
\begin{align*}
	W & := \{ u \in H_0^1(\Omega) \,\,| \,\, I(u)>0, J(u)<d \} \cup \{0\},\\
	\overline{W} & := \{ u \in H_0^1(\Omega) \,\,| \,\, I(u)\geq 0, J(u)\leq d \}.
\end{align*}
In addition, we define
\begin{align*}
	V_{\delta}&:= \{ u \in H_0^1(\Omega) \,\,| \,\, I_\delta(u)<0, J(u)<d(\delta) \},\\
	V &:=\{ u \in H_0^1(\Omega) \,\,| \,\, I(u)<0, J(u)<d \}.
\end{align*}

From the definition of $W_{\delta}$, $V_{\delta}$ and Lemma \ref{lem-5}, we derive the following properties:
\begin{lem}\label{lem-7}
	If $0<\delta'<\delta''\leq 1$ and $1\leq \delta''<\delta' <b $, then $W_{\delta'} \subset W_{\delta''}$ and $V_{\delta'} \subset V_{\delta''}$, respectively.
\end{lem}
\begin{lem}\label{lem-8}
	Let $0<J(u)<d$ for some $u \in H_0^1(\Omega)$. Let $\delta_1<\delta_2$ are two roots of equation $J(u)=d(\delta)$. Then the sign of $I_{\delta}(u)$ is unchangeable for $\delta_1<\delta<\delta_2$. 
\end{lem}
\begin{proof}[Proof of Lemma \ref{lem-8}]
	The proof follows from contradiction argument. $||\nabla u|| \neq 0$ comes from $J(u)>0$. If sign of $I_{\delta}(u)$ is changeable for $\delta_1<\delta <\delta_2$, then there exists $\delta^* \in (\delta_1,\delta_2)$ such that $I_{\delta^*}(u)=0$. Thus by the definition of $d(\delta)$ we have $J(u)\geq d(\delta^*)$ that contradicts 
	\begin{equation*}
		J(u)=d(\delta_1)=d(\delta_2) < d(\delta^*),
	\end{equation*}
this completes the proof.
\end{proof}
\section{Invariant sets and vacuum isolating of solutions}\label{Sec3}
In this section, the invariance sets and the vacuum isolating of solutions are discussed for semilinear wave equation \eqref{Wave-problem}. 

Recall the total energy of semilinear wave equation \eqref{Wave-problem}
\begin{equation*}
	E(t) = \frac{1}{2} || u_t||^2 + \frac{1}{2} || \nabla u ||^2 - \int_{\Omega } [F(u)- \sigma] dx \equiv \frac{1}{2} || u_t||^2 + J(u).
\end{equation*} 
\begin{thm}\label{thm-inv-sets}
	Let $f(u)$ satisfy $(H)$, $u_0(x)\in H_0^1(\Omega)$, and $u_1(x) \in L^2(\Omega)$. Assume that $0<e<d$, then equation $d(\delta)=e$ has two roots $\delta_1<\delta_2$. Therefore, we formulate the following properties
	\begin{itemize}
		\item[$(i)$] All solutions of problem \eqref{Wave-problem} with $E(0)=e$ belong to $W_{\delta}$ for $\delta_1 <\delta < \delta_2$, provided  $I(u_0)>0$ or $||\nabla u_0|| =0$.
		\item[$(ii)$] All solutions of problem \eqref{Wave-problem} with $E(0)=e$ belong to $V_{\delta}$ for $\delta_1 <\delta < \delta_2$, provided  $I(u_0)<0$.
	\end{itemize}
\end{thm}
\begin{proof}[Proof of Theorem \ref{thm-inv-sets}]
	In the case $(i)$, we suppose that $u(x,t)$ is any solution of problem \eqref{Wave-problem} with $E(0)=e$ and $I(u_0)>0$ or $||\nabla u_0|| =0$. If $I(u_0)>0$, then we have 
	\begin{equation}\label{eq-3.1}
		a(\delta) || \nabla u_0||^2 + \frac{1}{\alpha} I_{\delta}(u_0) < J(u_0) \leq d(\delta),
	\end{equation} 
	and
	\begin{equation}\label{eq-3.2}
		\frac{1}{2} || u_1||^2 + J(u_0) = E(0) = d(\delta_1)=d(\delta_2)< d(\delta), \,\,\, \delta_1 < \delta <\delta_2.
	\end{equation}
	From \eqref{eq-3.2} and the definition of $d(\delta)$ we get $I_{\delta}(u_0)>0$ and $J(u_0)<d(\delta)$ that is $u_0(x) \in W_{\delta}$ for $0<\delta<b$. 
	
	Now we prove $u(t) \in W_{\delta}$ for $\delta_1<\delta<\delta_2$ and $0<t<T$, where $T$ is the maximal existence time of $u(t)$. Arguing by contradiction, there must exist $t_0 \in (0,T)$ such that $u(t_0)\in \partial W_{\delta}$ for some $\delta \in (\delta_1, \delta_2)$, and $I_{\delta}(u(t_0))=0$, $|| \nabla u(t_0)|| \neq 0$ or $J(u(t_0))=d(\delta)$. From 
	\begin{equation}\label{eq-3.3}
		\frac{1}{2} || u_t||^2 + J(u) = E(0) < d(\delta), \,\,\, \delta \in (\delta_1,\delta_2), \,\, t \in (0,T),
	\end{equation}
it is easy to see that $J(u(t_0))\neq d(\delta)$. Also, if $I_{\delta}(u(t_0))=0$ and $|| \nabla u(t_0)|| \neq 0$, then by the definition of $d(\delta)$ we have $J(u(t_0))\geq d(\delta)$ that contradicts \eqref{eq-3.3}.

In the case $(ii)$, we assume again that $u(x,t)$ is any solution of problem \eqref{Wave-problem} with $E(0)=e$ and $I(u_0)<0$. As a previous case, $u_0(x)\in V_{\delta}$ can be established by Lemma \ref{lem-8} and \eqref{eq-3.2}. 

The proof of $u(t) \in V_{\delta}$ for $\delta \in (\delta_1,\delta_2)$ follows from arguing by contradiction. Let $t_0 \in (0,T)$ be the first time such that $u(t) \in V_{\delta}$ for $t \in [0,t_0)$ and $u(t_0) \in \partial V_{\delta}$, i.e. $I_{\delta}(u(t_0))=0$ or $J(u(t_0))=d(\delta)$ for some $\delta \in (\delta_1,\delta_2)$. From \eqref{eq-3.3} it follows that $J(u(t_0))\neq d(\delta)$. If $I_{\delta}(u(t_0))=0$, then $I_{\delta}(u(t))<0$ for $t \in (0,t_0)$ and Lemma \ref{lem-4} yield $|| \nabla u(t)||>r(\delta)$ and $|| \nabla u(t_0)||\geq r(\delta)$. Hence by the definition of $d(\delta)$ we have $J(u(t_0))\geq d(\delta)$ which contradicts \eqref{eq-3.3}.
\end{proof}
\begin{cor}\label{cor_1}
	If the assumption $E(0)=e$ in Theorem \ref{thm-inv-sets} is replaced by $0<E(0)\leq e$, then the conclusion of Theorem \ref{thm-inv-sets} also holds.
\end{cor}
\begin{thm}
	Let $f(u)$, $u_i(x)$ for $i=0,1$, $e$ and $(\delta_1,\delta_2)$ be the same as those in Theorem \ref{thm-inv-sets}. Then for any $\delta \in (\delta_1,\delta_2)$ both sets $W_{\delta}$ and $V_{\delta}$ are invariant, thereby both sets 
	\begin{equation*}
		W_{\delta_1\delta_2} = \bigcup _{\delta_1<\delta<\delta_2} W_{\delta}, \,\, \text{ and }\,\,\,  V_{\delta_1\delta_2} = \bigcup _{\delta_1<\delta<\delta_2} V_{\delta},
	\end{equation*} 
	are invariant respectively under the flow of \eqref{Wave-problem}, provided by $0<E(0)\leq e$. 
\end{thm}

Note that if $0<E(0)\leq e$, then $I(u_0)=0$ and $||\nabla u_0 ||\neq 0$ are impossible, which comes from Theorem  \ref{thm-inv-sets}. We see that for the set of all solutions of problem \eqref{Wave-problem} with $0<E(0)\leq e$ there exists a vacuum region 
\begin{equation*}
	U_e = \{ u \in H_0^1(\Omega) \,\, | \,\, || \nabla u|| \neq 0 \text{ and } I_{\delta}(u)=0, \,\,\, \delta_1 <\delta<\delta_2 \}
\end{equation*}
such that there is not any solution of problem \eqref{Wave-problem} in $U_e$. 

The vacuum region $U_e$ become bigger and bigger with decreasing of $e$. As the limit case we get
 \begin{equation*}
 	U_0 = \{ u \in H_0^1(\Omega) \,\, | \,\, || \nabla u|| \neq 0 \text{ and } I_{\delta}(u)=0, \,\,\, 0 <\delta<b \}.
 \end{equation*}

In the case $E(0)\leq 0$ we have the following invariant set of solutions for problem \eqref{Wave-problem}.
\begin{lem}\label{lem-9}
	All nontrivial solutions of problem \eqref{Wave-problem} with $E(0)=0$ satisfy
	\begin{equation*}
		|| \nabla u || \geq \left( \frac{1}{2AC_*^{\gamma}} \right)^{\frac{1}{\gamma-2}}.
	\end{equation*}
\end{lem}  
\begin{proof}[Proof of Lemma \ref{lem-9}]
	First, we get $J(u)\leq 0$ for $0\leq t< T$, where $T$ is the maximal existence time of $u(t)$, from the energy equality
	\begin{equation*}
		\frac{1}{2}|| u_t||^2 + J(u) = E(0) =0.
	\end{equation*}
Then 
\begin{align*}
	\frac{1}{2}|| \nabla u||^2 &\leq \int_{\Omega} [F(u)-\sigma]dx \leq A \int_{\Omega} |u|^{\gamma}dx\\
	&=A|| u||^{\gamma}_{\gamma} \leq AC_*^{\gamma} || \nabla u||^{\gamma-2}|| \nabla u||^2.
\end{align*}
This completes the proof.
\end{proof}
\begin{thm}\label{thm-E-neg}
	Let $u_0(x)\in H_0^1(\Omega)$ and $u_1(x)\in L^2(\Omega)$. Suppose that $E(0)=0$ or $E(0)<0$ with $||\nabla u_0|| \neq 0$. Then all solutions of problem \eqref{Wave-problem} belong to $V_{\delta}$ for $\delta \in (0,\alpha/2-\beta/\lambda_1 )$.
\end{thm}
\begin{proof}[Proof of Theorem \ref{thm-E-neg}]
	By using the conservation law of energy and for $\delta \in (0,\alpha/2- \beta/\lambda_1)$, we have
	\begin{equation*}
		J(u) \geq a(\delta)||\nabla u||^2  + \frac{1}{\alpha}I_{\delta}(u),
	\end{equation*}
	we write the following expression
	\begin{align}\label{eq-3.4}
		E(0) \geq \frac{1}{2}|| u_t||^2 + a(\delta)||\nabla u||^2  + \frac{1}{\alpha}I_{\delta}(u).
	\end{align}
	If $E(0)<0$, then \eqref{eq-3.4} gives $I_{\delta}(u)<0$ and $J(u)<0<d(\delta)$ for  $\delta \in (0,\alpha/2-\beta/\lambda_1 )$. If $E(0)=0$ and $|| \nabla u_0||\neq 0$, then Lemma \ref{lem-9} shows that $|| \nabla u||$ is positive for $t\in [0,T)$. So from \eqref{eq-3.4} we have $I_{\delta}(u)<0$ and $J(u)<0<d(\delta)$ for  $\delta \in (0,\alpha/2-\beta/\lambda_1 )$.
\end{proof}
\section{Global existence of solutions for semilinear wave equation}\label{Sec4}
In this section, we prove a global existence of solutions of semilinear wave equation \eqref{Wave-problem} by using the family of potential wells. 

\begin{defn}
	Let $u \in L^{\infty}(0,T;H_0^1(\Omega))$ and $u_t \in L^{\infty}(0,T;L^2(\Omega))$. Then a function $u(x,t)$ is called a {\em weak solution} of problem \eqref{Wave-problem} on $\Omega \times [0,T)$ if 
	\begin{equation*}
		(u_{t},v) + \int_0^t (\nabla u, \nabla v) d\tau = \int_0^t (f(u),v)d\tau + (u_1,v),
	\end{equation*}
	for every $v \in H_0^1(\Omega)$, $u(x,0)=u_0(x)\in H_0^1(\Omega)$, and $t \in (0,T)$. 
\end{defn}  
\subsection{Initial condition $ 0<E(0)<d$ and $I(u_0)>0$}
\begin{thm}\label{thm-global}
	Let $f(u)$ satisfy $(H)$. Let $u_0(x) \in H_0^1(\Omega)$, $u_1(x)\in L^2(\Omega)$ and $2\beta < \lambda_1(\alpha -2)$. Suppose that $0<E(0)<d$ and $I(u_0)>0$ or $|| \nabla u_0||=0$. Then problem \eqref{Wave-problem} admits a global weak solution $u(t) \in L^{\infty}(0,\infty; H_0^1(\Omega))$, $u_t(t)\in (0,\infty;L^2(\Omega)) $ and $u(t) \in W$ for $t \in [0,\infty)$.
\end{thm}
\begin{lem}[Lemma 1.3 in \cite{Lions}]\label{lem-Lions}
	Let $D$ be a bounded domain in $\mathbb{R}^n \times \mathbb{R}_t$, $g_{\mu}$ and $g$ are functions from $L^q(D)$, $1<q<\infty$, so 
	\begin{equation*}
	|| g_{\mu} ||_{L^q(D)} \leq C, \,\,\,\, g_{\mu} \rightarrow g \,\, \text{ in } D.
	\end{equation*}
	Then $g_{\mu} \rightarrow g $ weakly star in $L^q(D)$.
\end{lem}
\begin{proof}[Proof of Theorem \ref{thm-global}]
 Here we follow the proof from Liu-Zhao \cite{LZh-06}. Let $w_j(x)$ be a system of base functions in $H_0^1(\Omega)$. Then we are able to construct the approximate solutions $u_m(x,t)$ of problem \eqref{Wave-problem} as follows
 \begin{equation*}
 	u_m(x,t) = \sum_{j=1}^m g_{jm}(t)w_j(x) \,\, \text{ for } \,\, m=1,2,\ldots
 \end{equation*}
 satisfying
 \begin{equation}\label{eq-4.1}
 	(u_{m}'',w_s) + (\nabla u_m, \nabla w_s) = (f(u_m),w_s) \,\, \text{ for } \,\, s=1,2,\ldots,m,
 \end{equation}
 \begin{equation}\label{eq-4.2}
 	u_m(x,0) = \sum_{j=1}^m a_{jm}w_j(x) \rightarrow u_0(x) \,\, \text{ in } \,\, H_0^1(\Omega),
 \end{equation}
 \begin{equation}\label{eq-4.3}
 	u_{m}'(x,0) = \sum_{j=1}^m b_{jm}w_j(x)  \rightarrow u_1(x) \,\, \text{ in } \,\, L^2(\Omega).
 \end{equation}
Note that $u_t=u'$. If we multiply \eqref{eq-4.1} by $g'_{sm}(t)$ and sum over $s$, then this gives
\begin{equation*}
	(u_{m}'',u_{m}') + (\nabla u_m, \nabla u_{m}') = (f(u_m),u_{m}') \,\, \text{ for } \,\, m=1,2,\ldots.
\end{equation*}
This implies conservation of energy $E_m(t)$, which allows to write for $0\leq t < \infty$
 \begin{align}\label{eq-4.4}
 	E(t)&=\frac{1}{2} || u_{m}'||^2 + \frac{1}{2} || \nabla u_m||^2 - \int_{\Omega }[F(u_m)-\sigma] dx \nonumber\\
 	&=\frac{1}{2} || u_{m}'||^2 +J(u_m) = E_m(0) < d, 
 \end{align}
 and for $I(u_{m}(x,0))>0$ we have $u_m \in W$ for sufficiently large $m$ and $0\leq t<\infty$. 
 
 Now we obtain a priori estimates. Let 
 \begin{align*}
 	J(u_m) &= \frac{1}{2} || \nabla u_m||^2 - \int_{\Omega } [F(u_m)- \sigma] dx \\
 		   &\geq \frac{1}{2} || \nabla u_m||^2 - \frac{1}{\alpha} \int_{\Omega } u_m f(u_m) dx - \frac{\beta}{\alpha} || u_m ||^2  \\
 		   & \geq \left( \frac{1}{2} - \frac{1}{\alpha} - \frac{\beta}{\alpha \lambda_1} \right)||\nabla u_m||^2 + \frac{1}{\alpha} I(u_m) \\
 		   & \geq a ||\nabla u_m||^2,
 \end{align*}
 where $a =  \frac{1}{2} - \frac{1}{\alpha} - \frac{\beta}{\alpha \lambda_1}$. Then from \eqref{eq-4.4} it follows that
 \begin{equation}\label{eq-4.5}
 	\frac{1}{2} || u_{m}'||^2 + a||\nabla u_m||^2 < d, \,\,\, 0\leq t <\infty,
 \end{equation}
 for sufficiently large $m$. Then for $0\leq t <\infty $ \eqref{eq-4.5} implies the following estimates 
 \begin{align}
 &	|| \nabla u_m ||^2 < \frac{d}{a}, \label{eq-4.6}\\
 &	|| u_m||^2_{\gamma} \leq C_*^2|| \nabla u_m||^2 <C_*^2\frac{d }{a},	\label{eq-4.7} \\
& || u_{m}'||^2 < 2d , \label{eq-4.8}
 \end{align}
 and Lemma \ref{lem-1} and \eqref{eq-4.7} give
\begin{align}
	  || f(u_m)||^q_q &\leq \int_{\Omega } [\gamma A |u_m|^{\gamma-1}]^q dx = \gamma^q A^q ||  u_m||^{\gamma}_{\gamma} \nonumber \\
& < \gamma^q A^q C_*^{\gamma} \left(\frac{d}{a}\right)^{\gamma/2},
\end{align}
where $q= \frac{\gamma}{\gamma-1}$.

Since a priori estimates \eqref{eq-4.6}-\eqref{eq-4.8}, we may extract the subsequence $\{u_l\}$ from sequence $\{u_m\}$ such that 
\begin{align*}
	u_l &\rightarrow u \,\, \text{ weakly star in }  L^{\infty}(0,\infty;H_0^1(\Omega)) \,\, \text{ and a.e. } Q=\Omega \times [0,\infty) \\
	u_l'&\rightarrow u_t \,\, \text{ weakly star in }  L^{\infty}(0,\infty;L^2(\Omega)),\\
	f(u_l)&\rightarrow \chi \,\, \text{  in }  L^{\infty}(0,\infty;L^q(\Omega)) \,\,\, \text{ and a.e. } Q=\Omega \times [0,\infty).
\end{align*}
From Lemma \ref{lem-Lions} we have $\chi=f(u)$. We integrate equation \eqref{eq-4.1} with respect to time from $0$ to $t$ as follows
\begin{equation}\label{eq-4.10}
	(u_m',w_s) -(u_m'(0),w_s) + \int_0^t(\nabla u_m,\nabla w_s) d\tau = \int_0^t (f(u_m),w_s) d\tau. 
\end{equation}
For fixed $s$ in \eqref{eq-4.10} and taking $m=l\rightarrow \infty$ we get 
\begin{equation*}
(u_t,w_s) -(u_1,w_s) + \int_0^t(\nabla u,\nabla w_s) d\tau = \int_0^t (f(u),w_s) d\tau, \,\,\, \forall s, 
\end{equation*}
then 
\begin{equation*}
	(u_t,v) -(u_1,v) + \int_0^t(\nabla u,\nabla v) d\tau = \int_0^t (f(u),v) d\tau, \,\,\, \forall v \in H_0^1(\Omega), \,\, t>0.
\end{equation*}
Also \eqref{eq-4.2} implies that $u_0(x)$ in $H_0^1(\Omega)$. Therefore $u(x,t)$ is a global weak solution of problem \eqref{Wave-problem}.
\end{proof}
\begin{thm} Let $f(u)$ satisfy $(H)$.
	Let $u_0(x)\in H_0^1(\Omega)$ and $u_1(x)\in L^2(\Omega)$. The problem \eqref{Wave-problem} admits a global weak solution $u(t)\in L^{\infty}(0,\infty;H_0^1(\Omega))$ with $u_t(t)\in L^{\infty}(0,\infty;L^2(\Omega)) $ and
	\begin{itemize}
		\item $u(t)\in W_{\delta}$ for $\delta_1<\delta<\delta_2$,
		or
		\item $|| \nabla u ||^2 \leq E(0)/d(\delta_1)$, $||u_t ||^2\leq 2E(0)$, 
	\end{itemize}  if the following respective assumptions hold
\begin{itemize}
	\item $0<E(0)<d$ and $I_{\delta_2}<0$ or
	\item $0<E(0)<d$ and $|| \nabla u_0||<r(\delta_2)$.
\end{itemize} 
\end{thm}
\subsection{Critical initial condition $E(0)=d$ and $I(u_0)\geq 0$} Here we consider the global existence of solution for problem \eqref{Wave-problem} with critical initial condition $E(0)=d$ and $I(u_0)\geq 0$.
\begin{thm}\label{thm-E=d}
	Let $f(u)$ satisfy $(H)$. Let $u_0 \in H_0^1(\Omega)$ and $u_1(x)\in L^2(\Omega)$. Suppose that $E(0)=d$ and $I(u_0)\geq 0$. Then problem \eqref{Wave-problem} admits a hlobal weak solution $u(t) \in L^{\infty}(0,\infty;H_0^1(\Omega))$ with $u_t(t)\in L^{\infty}(0,\infty; L^2(\Omega))$ and $u(t)\in \overline W$ for $0\leq t <\infty$.
\end{thm}
\begin{proof}[Proof of Theorem \ref{thm-E=d}]
	The proof is divided for two cases $|| \nabla u_0||\neq0$  and $|| \nabla u_0|| =0$.

In the case $|| \nabla u_0||\neq 0$, we take $\epsilon_m = 1 - 1/m$ and $u_{0m}=\epsilon_m u_0$ for $m=2,3,\ldots$. Then we consider problem \eqref{Wave-problem} with the following initial conditions
		\begin{equation}\label{IC-m}
			u(x,0) = u_{0m}(x) \,\, \text{ and } \,\, u_t(x,0) = u_1(x).
		\end{equation}
The fact $\epsilon^*=\epsilon^*(u_0)\geq 1$ follows from $I(u_0)\geq 0$ and Lemma \ref{lem-3}. Then for $I(u_{0m})>0$ we have 
\begin{equation}
	J(u_{0m}) \geq \left(\frac{1}{2} -\frac{1}{\alpha} - \frac{\beta}{\alpha  \lambda_1}\right)|| \nabla u_{0m}||^2 +\frac{1}{\alpha} I(u_{0m})>0,
\end{equation}		
and $J(u_{m0})<J(u_0)$. Then 
\begin{equation*}
	0<E_m(0) = \frac{1}{2}|| u_1||^2 + J(u_{0m}) < \frac{1}{2}|| u_1||^2 + J(u_{0}) =E(0)=d.
\end{equation*}
Now it is easy to see from Theorem \ref{thm-global} that for each $m$ problem \eqref{Wave-problem} with initial conditions \eqref{IC-m} admits a global weak solution $u_m(t)\in L^{\infty}(0,\infty;H_0^1(\Omega))$ with $u_{mt}(t)\in L^{\infty}(0,\infty; L^2(\Omega))$ and $u_m(t)\in \overline W$ for $0\leq t<\infty$ satisfying 
\begin{equation}\label{eq-5.2}
	(u_{mt},v) + \int_0^t(\nabla u_m,\nabla v) d\tau = \int_0^t (f(u_m),v)d\tau +(u_1,v),
\end{equation}
for every $v \in H_0^1(\Omega)$, and
\begin{equation}\label{eq-5.3}
	\frac{1}{2} || u_{mt} ||^2 + J(u_m) = E_m(0) <d.
\end{equation}
Then we obtain the following estimates
\begin{align*}
&	|| \nabla u_m ||^2 < \frac{d}{a},\\
&	|| u_m||^2_{\gamma} \leq C_*^2|| \nabla u_m||^2 <C_*^2\left(\frac{d}{a}\right), \\
& || u_{m}'||^2 < 2d ,
\end{align*}
and
\begin{align*}
|| f(u_m)||^q_q < \gamma^q A^q C_*^{\gamma} \left(\frac{d}{a}\right)^{\gamma/2},
\end{align*}
 where $a = \frac{1}{2} -\frac{1}{\alpha} - \frac{\beta}{\alpha  \lambda_1}$ and $q= \frac{\gamma}{\gamma-1}$.

From above priori estimates, we extract the subsequence $\{u_l\}$ from sequence $\{u_m\}$ such that as $l \rightarrow \infty$
\begin{align*}
u_l &\rightarrow u \,\, \text{ weakly star in }  L^{\infty}(0,\infty;H_0^1(\Omega)) \,\, \text{ and a.e. } Q=\Omega \times [0,\infty) \\
u_l'&\rightarrow u_t \,\, \text{ weakly star in }  L^{\infty}(0,\infty;L^2(\Omega)),\\
f(u_l)&\rightarrow f(u) \,\, \text{  in }  L^{\infty}(0,\infty;L^q(\Omega)) \,\,\, \text{ and a.e. } Q=\Omega \times [0,\infty).
\end{align*}
Letting $m=l\rightarrow \infty$ in \eqref{eq-5.2} we arrive at
\begin{equation}
(u_{t},v) + \int_0^t(\nabla u,\nabla v) d\tau = \int_0^t (f(u),v)d\tau +(u_1,v),
\end{equation}
for every $v \in H_0^1(\Omega)$ and $0\leq t <\infty$, which implies that $u(x,t)$ is the global weak solution of problem \eqref{Wave-problem}.

In the case $|| \nabla u_0|| =0$, we have 
\begin{equation*}
	\frac{1}{2} || u_1||^2 + J(u_0) = 	\frac{1}{2} || u_1||^2 =E(0)=d.
\end{equation*}
Let $\epsilon_m = 1- 1/m$ such that $u_{1m}(x) = \epsilon_mu_1(x)$ for $m=2,3,\ldots$. We consider problem \eqref{Wave-problem} with the following initial conditions
\begin{equation}\label{IC-m1}
	u(x,0) = u_0(x) \,\, \text{ and } \,\, u_t(x,0) = u_{1m}(x).
\end{equation}
From $|| \nabla u_0|| =0$, we get 
\begin{equation*}
0<E_m(0)=\frac{1}{2} || u_{1m}||^2 + J(u_0) = 	\frac{1}{2} || u_{1m}||^2 <E(0)=d.
\end{equation*}
As the previous case from Theorem \ref{thm-global} follows that for each $m$ problem \eqref{Wave-problem} with initial conditions \eqref{IC-m1} admits a global weak solution $u_m(t)\in L^{\infty}(0,\infty;H_0^1(\Omega))$ with $u_{mt}(t)\in L^{\infty}(0,\infty; L^2(\Omega))$ and $u_m(t)\in \overline W$ for $0\leq t<\infty$ satisfying \eqref{eq-5.2} and   \eqref{eq-5.3}. The rest of proof repeats the previous one.
\end{proof}
\section{Blow-up solutions of semilinear wave equation}\label{Sec5}
\subsection{Initial condition $E(0)<d$ and $I(u_0)<0$}
\begin{thm}\label{thm-blow}
Let $f$ satisfy $(H)$ with $2\beta < \lambda_1 (\alpha-2)$ and $\lambda_1$ is a first eigenvalue of Laplacian. Let $u_0(x)\in H_0^1(\Omega)$ and $u_1(x)\in L^2(\Omega)$. Suppose that energy $E(0)<d$ and $I(u_0)<0$. Then the solution of problem \eqref{Wave-problem} blows up such that 
\begin{equation}
    \lim_{t\rightarrow T} \int_{\Omega} u^2 dx = + \infty,
\end{equation}
in a finite time
\begin{equation}
    T = \frac{\int_{\Omega} u_0^2(x) dx}{(\alpha+2) \int_{\Omega} u_0(x)u_1(x) dx}.
\end{equation}
\end{thm}
\begin{proof}[Proof of Theorem \ref{thm-blow}]
	Define 
	\begin{align*}
	M(t) &= \int_{\Omega} u^2 dx.
	\end{align*}
	Then we have 
	\begin{align}
		M'(t)= 2 \int_{\Omega} uu_t dx \,\, \text{ and } \,\,M''(t)= 2\int_{\Omega} u_t^2 dx + 2\int_{\Omega} uu_{tt} dx.
	\end{align}
	Then by using $u_{tt} = \Delta u +f(u)$ and \eqref{cont_1} we have
	\begin{align*}
		M''(t) &= 2\int_{\Omega} u_t^2 dx + 2\int_{\Omega } [ \Delta u +f(u) ]u dx\\
		& = 2\int_{\Omega} u_t^2 dx - 2I(u) \\
		& \geq 2\int_{\Omega} u_t^2 dx +2\alpha \int_{\Omega } [F(u) - \sigma] dx - 2\beta \int_{\Omega } u^2 dx - 2\int_{\Omega } |\nabla u|^2dx.
	\end{align*}
	Then from the conservation of energy 
	\begin{equation*}
		E(0) = \frac{1}{2} \int_{\Omega } u_t^2 dx + \frac{1}{2} \int_{\Omega } |\nabla u|^2 dx- \int_{\Omega } [F(u)-\sigma ]dx,
	\end{equation*}
	and the Poincar\'e inequality, we get
	\begin{equation*}
		M''(t) \geq (\alpha +2)\int_{\Omega} u_t^2 dx + \left(\lambda_1(\alpha -2)- 2\beta\right) M(t) - 2\alpha E(0),
	\end{equation*}
	where $\lambda_1$ is the first eigenvalue of Laplacian and $2\beta > \lambda_1(\alpha -2)$. 
	
	In the case $E(0)<0$ and $2\beta > \lambda_1(\alpha -2)$, then 
	\begin{equation}\label{eq-M}
			M''(t) \geq (\alpha +2)\int_{\Omega} u_t^2 dx.
	\end{equation}
	
	In the case $0<E(0)<d$ and $2\beta < \lambda_1(\alpha -2)$, there exists sufficient large time $t_1$ such that 
	\begin{equation*}
		\left(\lambda_1(\alpha -2)- 2\beta\right) M(t) - 2\alpha E(0)>0.
	\end{equation*}
	Since $M(t)$ is a convex function of $t$, then we can find a time $t_1$ such that $M'(t_1)>0$, then $M(t)$ is increasing for all $t>t_1$. Also this part of proof can be shown by using the family of potential wells $V_{\delta}$ and Lemma \ref{lem-4} (see \cite{LZh-06}). 
	
	We are able to prove that \eqref{eq-M} for two cases $E(0)<0$ and $0<E(0)<d$. Next we obtain the blow-up solutions by using concavity argument and \eqref{eq-M} as follows
	\begin{equation*}
	M''(t)M(t) - (\alpha+3) [M'(t)]^2 \geq (\alpha+3)\left[\left( \int_{\Omega }u^2dx\right) \left(\int_{\Omega }u_t^2 dx\right) - \left( \int_{\Omega } uu_tdx\right)^2 \right]\geq 0,
	\end{equation*}
	which gives for $t\geq 0$  
	\begin{equation*}
	\frac{d}{dt} \left[ \frac{M'(t)}{M^{\alpha+3}(t)} \right] >0  \Rightarrow 	\begin{cases}
	M'(t) \geq \left[ \frac{M'(0)}{M^{\alpha+3}(0)} \right] M^{\alpha+3}(t),\\
	M(0)=\int_{\Omega} u_0^2 dx.
	\end{cases}
	\end{equation*}
	Then we have
	\begin{equation*}
	M(t) \geq M(0)\left( 1-\frac{ (\alpha+2) M'(0) }{M(0)} t\right)^{-\frac{1}{\alpha+2}}.
	\end{equation*}
	Then the blow-up time $T$ satisfies 
	\begin{equation*}
	0<T\leq \frac{M(0)}{(\alpha+2) M'(0)}.
	\end{equation*}
	This completes the proof.
\end{proof}
\subsection{Critical initial condition $E(0)=d$ and $I(u_0)<0$}
Next we prove the blow-up solution of problem \eqref{Wave-problem} with the critical initial conditions $E(0)=d$ and $I(u_0)<0$. First we need the following lemma.
\begin{lem}\label{lem-10}
	Let $u_0(x) \in H_0^1(\Omega)$ and $u_1(x) \in L^2(\Omega)$. Suppose that $E(0)=d$ and $(u_0,u_1)>0$. Then set
	\begin{equation*}
		V' = \{ u \in H_0^1(\Omega)\, : \, I(u)<0 \}
	\end{equation*}
	is invariant under the flow of problem \eqref{Wave-problem}.
\end{lem} 
\begin{proof}[Proof of Lemma \ref{lem-10}]
	Let $u(t)$ be any weak solution of problem \eqref{Wave-problem} with $E(0)=d$, $I(u_0)<0$ and $\int_{\Omega }u_0u_1dx\geq 0$, and $T$ is the maximal existence time of $u(t)$. $I(u(t))<0$ for $t \in (0,T)$ is established by a contradiction. Let it be false, then there exists $t_0 \in (0,T)$ such that $I(u(t_0))=0$ and $I(u(t))$ for $t\in [0,t_0)$. By Lemma \ref{lem-4}, we have $|| \nabla u(t)||>r(1)$ for $t \in [0,t_0)$ and $|| \nabla u(t_0)||\geq r(1)>0$. Next we have $J(u(t_0))=d$ and $\int_{\Omega } u_t^2(t_0) dx=0$ from the definition of $d$ which implies $J(u(t_0))\geq d$ and 
	\begin{equation*}
		\frac{1}{2} \int_{\Omega } u_t^2(t_0) dx + J(u(t_0)) \leq E(0) =d. 
	\end{equation*}
	Let 
	\begin{equation*}
		M(t) = \int_{\Omega } u^2(t)dx \,\, \text{ and } \,\, M'(t) = 2\int_{\Omega }u_t(x,t)dx,  
	\end{equation*}
	with 
	\begin{align*}
		M'(0) &= 2 \int_{\Omega } u_0(x)u_1(x) dx, \\
		M''(t)&= 2 \int_{\Omega }u_t^2(t)dx - 2I(u(t)) > 0 \,\, \text{ for } \,\, 0\leq t<t_0. 
	\end{align*}
	Hence $M'(t)$ is strictly increasing with respect to $t \in [0,t_0]$ along with $M'(0)\geq 0$ gives 
	\begin{equation*}
		M'(t_0) = 2\int_{\Omega } u(t_0) u_t(t_0)dx > 0.
	\end{equation*}
	This clearly contradicts $\int_{\Omega } u_t^2(t_0) dx=0$. This completes the proof. 
\end{proof}
	\begin{thm}\label{thm-blow-up-E=d}
		Let $f$ satisfy $(H)$ with $2\beta < \lambda_1 (\alpha-2)$ and $\lambda_1$ is a first eigenvalue of Laplacian. Let $u_0(x) \in H_0^1(\Omega)$ and $u_1(x) \in L^2(\Omega)$. Suppose that 
		\begin{align*}
			E(0)=d, \,\, I(u_0) <0, \,\, \text{ and } \int_{\Omega } u_0(x)u_1(x)dx \geq 0.
		\end{align*}
	Then the weak solution of problem \eqref{Wave-problem} blows up in a finite time.	
	\end{thm}
\begin{proof}[Proof of Theorem \ref{thm-blow-up-E=d}]
	Let 
	\begin{align*}
		M(t) = \int_{\Omega } u^2dx, \,\, \text{ and }\,\, M'(t) = 2\int_{\Omega } uu_tdx.
	\end{align*}
	As in previous case we have
	\begin{equation*}
	M''(t) \geq (\alpha +2)\int_{\Omega} u_t^2 dx + \left(\lambda_1(\alpha -2)- 2\beta\right) M(t) - 2\alpha d,
	\end{equation*}
	and
	\begin{equation}
		M''(t) = 2 \int_{\Omega } u^2_t dx - 2I(u)>0,
	\end{equation}
	for $t \in [0,\infty)$ because of $\int_{\Omega } u_0(x)u_1(x)dx \geq 0$ and Lemma \ref{lem-10}. $M''(t) >0$ means that $M'(t)$ is strictly increasing for $t\in [0,\infty)$. Hence for any $t_0>0$ we have $M'(t)\geq M'(t_0)>0$ for $t\geq t_0$,
	\begin{equation*}
		M(t) \geq M'(t_0)(t-t_0) +M(t_0) \geq M'(t_0)(t-t_0). 
	\end{equation*}
Therefore there exists sufficiently large time $t$ such that  
\begin{equation*}
	\left(\lambda_1(\alpha -2)- 2\beta\right) M(t) > 2\alpha d,
\end{equation*}
for $2\beta<\lambda_1(\alpha -2)$. Then we have
\begin{equation*}
	M''(t)\geq (\alpha+2)\int_{\Omega }u^2_tdx.
\end{equation*}
The rest of proof repeats the same computation as in Theorem \ref{thm-blow}.
\end{proof}
\section{Conclusion}\label{Sec6}
In conclusion, we arrive at the following results of global existence and blow-up of solutions for initial-boundary problem \eqref{Wave-problem}:
\begin{thm}
	Let $f(u)$ satisfy $(H)$, $u_0(x)\in H_0^1(\Omega)$, $u_1(x)\in L^2(\Omega)$. 
	Then problem \eqref{Wave-problem} admits a global weak solution if the following assumptions hold:
	\begin{itemize}
		\item $E(0)<d$ and $I(u_0)\geq0$;
		\item $E(0)=d$, $(u_0(x),u_1(x))\geq 0$, and $I(u_0)\geq0$.
	\end{itemize}
The solution of problem \eqref{Wave-problem} blows up in finite time if the following assumption hold:
	\begin{itemize}
	\item $E(0)<d$ and $I(u_0)<0$;
	\item $E(0)=d$, $(u_0(x),u_1(x))\geq 0$, and $I(u_0)<0$.
\end{itemize}
\end{thm}

\end{document}